\RequirePackage{fix-cm}
\documentclass{scrartcl}
\usepackage{graphicx}
    \usepackage[latin1]{inputenc}
    \usepackage[T1]{fontenc}

    \usepackage{amssymb, amsmath, amsfonts, amsthm} 
    \usepackage{mathtools} \mathtoolsset{showonlyrefs} 

    \usepackage{hyperref} 
    \usepackage[all]{hypcap}

    \hypersetup{colorlinks=false} 

    \newcommand{\N}{{\mathbb N}}
    \newcommand{\Z}{{\mathbb Z}}
    
    \newcommand{\R}{{\mathbb R}}
    \newcommand{\C}{{\mathbb C}}
    
    \newcommand{\BigO}{{O}}
    
    \newcommand{\SL}{\operatorname{SL}}
    
    \newcommand{\vol}{\operatorname{vol}}

    \let\op=\operatorname
    \newcommand{\mat}[1]{\begin{pmatrix}#1\end{pmatrix}}
    \newcommand{\e}{\varepsilon}

    \newcommand{\abs}[1]{\left|#1\right|}

    \renewcommand{\epsilon}{\varepsilon}

    \renewcommand{\Re}{\operatorname{Re}}

    \newcommand{\defeq}{\mathrel{\mathop:}=} 
    \newcommand{\eqdef}{=\mathrel{\mathop:}} 
    \newcommand{\inv}{^{-1}}
    \newcommand{\T}{^{t}} 
    \newcommand{\tr}{\operatorname{tr}}
    
    \newcommand{\dfn}[1]{\textbf{#1}}

    \newcommand{\dd}[1]{\mathop{d#1}}

    \newcommand{\minfelterm}{\BigO_\e(T^{n(n-1)-1/(2n)+\epsilon})}

    \theoremstyle{definition}
    
    \newtheorem{Remark}[equation]{Remark}

    \theoremstyle{plain}
    \newtheorem{Theorem}[equation]{Theorem}
    \newtheorem{Lemma}[equation]{Lemma}
    \newtheorem{Proposition}[equation]{Proposition}
    \newtheorem{Corollary}[equation]{Corollary}

\begin{document}

\title{Counting nonsingular matrices with primitive row vectors
  \thanks{The author was partially supported by the Swedish Research Council.}
}


\author{Samuel Holmin
}




\maketitle

\begin{abstract}
We give an asymptotic expression for the number of nonsingular integer $n\times n$-matrices with primitive row vectors, determinant $k$, and Euclidean matrix norm less than $T$, as $T\to\infty$. 

We also investigate the density of matrices with primitive rows in the space of matrices with determinant $k$, and determine its asymptotics for large $k$.
\end{abstract}

\section{Introduction}
An integer vector $v\in\Z^n$ is \dfn{primitive}\index{primitive vector} if it cannot be written as an integer multiple $m\neq 1$ of some other integer vector $w\in \Z^n$. Let $A$ be an integer $n\times n$-matrix with nonzero determinant $k$ and primitive row vectors. We ask how many such matrices $A$ there are of {Euclidean norm} at most $T$, that is, $\|A\|\leq T$, where $\|A\|\defeq\sqrt{\sum a_{ij}^2}=\sqrt{\tr(A\T A)}$. Let $N'_{n,k}(T)$ be this number (the prime in the notation denotes the primitivity of the rows), and let $N_{n,k}(T)$ be the corresponding counting function for matrices with not necessarily primitive row vectors. We will determine the asymptotic behavior of $N'_{n,k}(T)$ for large $T$, and investigate the density $D_n(k)\defeq \lim_{T\to\infty} N'_{n,k}(T)/N_{n,k}(T)$ of matrices with primitive vectors in the space of matrices with nonzero determinant $k$. Since $N'_{n,k}$ and $N_{n,k}$ do not depend on the sign of $k$, we will without loss of generality assume that $k>0$.

Let $M_{n,k}$\index{$M_{n,k}$} be the set of integer $n\times n$-matrices with determinant $k$. Then $N_{n,k}(T) = |B_T\cap M_{n,k}|,$ where $B_T$\index{$B_T$} is the (closed) ball of radius $T$ centered at the origin in the space $M_n(\R)$\index{$M_n(\R)$} of real $n\times n$-matrices  equipped with the Euclidean norm. Throughout, we will assume that $n\geq 2$ and $k>0$ unless stated otherwise.

Duke, Rudnick and Sarnak \cite{DRS} found that the asymptotic behavior of $N_{n,k}$\index{$N_{n,k}(T)$} is given by
\begin{gather}
  N_{n,k}(T)=c_{n,k}T^{n(n-1)}+\BigO_\e(T^{n(n-1)-1/(n+1)+\e}),
\end{gather}
as $T\to\infty$, for a certain constant $c_{n,k}$\index{$c_{n,k}$} and all $\e>0$, where the error term  can be improved to $O(T^{4/3})$ for $n=2$. 
The corresponding case for singular matrices was later investigated by Katznelson, who proved in \cite{Katznelson} that
\begin{gather}
  N_{n,0}(T)=c_{n,0}T^{n(n-1)}\log T+\BigO(T^{n(n-1)}).
\end{gather}
See the next page for the constants $c_{n,k}$ and $c_{n,0}$.

Let $M_{n,k}'$\index{$M'_{n,k}$}\index{$N'_{n,k}(T)$}\index{$c'_{n,k}$} be the set of matrices in $M_{n,k}$ with primitive row vectors. Then $N'_{n,k}(T)= |B_T\cap M'_{n,k}|$. Wigman \cite{Wigman} determined the asymptotic behavior of the counting function $|G_T\cap M'_{n,0}|$, where $G_T$ is a ball of radius $T$ in $M_n(\R)$, under a slightly different norm than ours. The results can be transferred to our setting, whereby we have
\begin{align}
N_{n,0}'(T)&=c'_{n,0}T^{n(n-1)}\log T+\BigO(T^{n(n-1)}),\qquad n\geq 4,\\
N_{3,0}'(T)&=c'_{3,0}T^{3(3-1)}\log T+\BigO(T^{3(3-1)}\log\log T),\\
N_{2,0}'(T)&=c'_{2,0}T^{2(2-1)}+\BigO(T).
\end{align}

The case $n=2$ above is equivalent to the \dfn{primitive circle problem}\index{primitive circle problem}, which asks how many primitive vectors there are of length at most $T$ in $\Z^2$ given any (large) $T$.

The main result in our paper is the following asymptotic expression for the number of nonsingular matrices with primitive row vectors and fixed determinant.
\begin{Theorem} \label{main_theorem}
Let $k\neq 0$. Then
\begin{gather}
  N'_{n,k}(T)=c'_{n,k}T^{n(n-1)}+\minfelterm,
\end{gather}
as $T\to\infty$ for a certain constant $c'_{n,k}$ and all $\e>0$.
\end{Theorem}
\noindent Section \ref{counting_problem} is dedicated to the proof of this theorem. 

The constant in Theorem \ref{main_theorem} can be written as
\[c'_{n,k}=\dfrac{C_1}{\abs{k}^{n-1}}\sum_{d_1\cdots d_n=|k|}\prod_{i=1}^n \sum_{g|d_i}\mu(g)\left(\dfrac {d_i}g\right)^{i-1},\]
for $k\neq 0$, which may be compared to the constants obtained from \cite{DRS}, \cite{Katznelson} and \cite{Wigman}, namely
\begin{align}
c_{n,k}&=\dfrac{C_1}{\abs{k}^{n-1}}\sum_{d_1\cdots d_n=|k|} \prod_{i=1}^n d_i^{i-1}\\
c_{n,0}&=C_0\dfrac{n-1}{\zeta(n)}\\
c'_{n,0}&=\begin{cases}\displaystyle C_0 \dfrac{n-1}{\zeta(n-1)^n\zeta(n)}& (n\geq 3)\\
  \displaystyle \frac{\pi T^2}{\zeta(2)}& (n=2)
\end{cases}
\end{align}
where $\zeta$ is the Riemann zeta function, $\mu$ is the Möbius function, and $C_0$ and $C_1$ are constants defined as follows (these depend on $n$, but we will always regard $n$ as fixed). Let $\nu$ be the normalized Haar measure on $\SL_n(\R)$. The measure $w$ below is obtained by averaging the $n(n-1)$-dimensional volume of $E\cap A_u$ over all classes $A_u\defeq \{A\in M_n(\R): Au=0\}$ for nonzero $u\in \R^n$. In Appendix \ref{section_katznelson_measure} we give a precise definition of $w$ and calculate $w(B_1)$. 

Write $V_n$ for the volume of the unit ball in $\R^n$ and $S_{n-1}$ for the surface area of the $(n-1)$-dimensional unit sphere in $\R^n$. Then\index{$C_0$}\index{$C_1$}
\begin{align}
C_0&\defeq w(B_1)=\dfrac{V_{n(n-1)}S_{n-1}}2=\dfrac{\pi^{n^2/2}}{\Gamma\left(\dfrac n2\right)\Gamma\left(\dfrac{n(n-1)}2+1\right)},\\
C_1&\defeq \lim_{T\to\infty}\dfrac{\nu(B_T\cap \SL_n(\R))}{T^{n(n-1)}} =\dfrac{V_{n(n-1)}S_{n-1}}{2\zeta(2)\cdots\zeta(n)} = \frac{C_0}{\zeta(2)\cdots\zeta(n)}.
\end{align}

\subsection{Density}
It will be interesting to compare the growth of $N'_{n,k}$ to that of $N_{n,k}$. 
We define the \dfn{density} of matrices with primitive rows in the space $M_{n,k}$ to be
\begin{gather}
  D_n(k)\defeq 
    \lim_{T\to\infty}\dfrac{N'_{n,k}(T)}{N_{n,k}(T)}=\dfrac{c'_{n,k}}{c_{n,k}}.
\end{gather}

The asymptotics of $N_{n,0}$ and $N'_{n,0}$ are known from \cite{Katznelson} and \cite{Wigman}, and taking their ratio, we see that
\begin{gather}
  D_n(0)=\dfrac1{\zeta(n-1)^n} 
\end{gather}
for $n\geq 3$. 
We will be interested in the value of $D_n(k)$ for large $n$ and large $k$. The limit of $D_n(k)$ as $k\to\infty$ does not exist, but it does exist for particular sequences of $k$.

We say that a sequence of integers is \dfn{totally divisible}\index{totally divisible sequence} if its terms are eventually divisible by all positive integers smaller than $m$, for any $m$. We say that a sequence of integers is \dfn{rough}\index{rough sequence} if its terms eventually have no divisors smaller than $m$ (except for $1$), for any $m$.
An equivalent formulation is that a sequence $(k_1,k_2,\ldots)$ is totally divisible if and only if $|k_i|_p\to 0$ as $i\to\infty$ for all primes $p$, and $(k_1,k_2,\ldots)$ is rough if and only if $|k_i|_p\to 1$ as $i\to\infty$ for all primes $p$, where $|m|_p$ denotes the $p$-adic norm of $m$.

We state our main results about the density $D_n$. We prove these in section \ref{chapter_density}.

\begin{Theorem} \label{thm_density}
  Let $n\geq 3$ be fixed. Then $D_n$ is a multiplicative function, and $D_n(p^m)$ is strictly decreasing as a function of $m$ for any prime $p$. We have
  \begin{gather}
    \dfrac1{\zeta(n-1)^n} = D_n(0) < D_n(k) < D_n(1) = 1
  \end{gather}
  for all $k\neq 0,1$. Now let $k_1,k_2,\ldots$ be a sequence of integers. Then \[D_n(k_i)\to 1\] if and only if $(k_1,k_2,\ldots)$ is a rough sequence, and \[D_n(k_i)\to \dfrac1{\zeta(n-1)^n}\] if and only if $(k_1,k_2,\ldots)$ is a totally divisible sequence. Moreover, $D_n(k)\to 1$ uniformly as $n\to\infty$.
\end{Theorem}

\begin{Remark} \label{partial_limits_remark}
    Given an integer sequence $k_1,k_2,\ldots$, write $k_i=\pm\prod_p p^{m_p(i)}$ for the prime decomposition of $k_i$ for nonzero $k_i$, and otherwise formally define $m_p(i)=\infty$ for all $p$ if $k_i$ is zero. For $n\geq 3$, it follows from Theorem \ref{thm_density} that the limit $\lim_{i\to\infty} D_n(k_i)$ exists and is equal to
  $
      \prod_{p}\lim_{i\to\infty}D_n(p^{m_p(i)})
  $
  where the product extends over all primes $p$, whenever every sequence of prime exponents $(m_p(1), m_p(2),\ldots)$ is either eventually constant or tends to $\infty$.
\end{Remark}

We prove Theorem \ref{thm_density} for nonzero $k_i$, but it is interesting that this formulation holds for $k=0$ also. The case of $k=0$ was proved by Wigman \cite{Wigman}, where he found that $D_n(0)$ equals $1/\zeta(n-1)^n$. 
We remark that Theorem \ref{thm_density} implies that 
\[D_n(k_i)\to D_n(0)\]
if and only if $(k_1,k_2,\ldots)$ is totally divisible, for any fixed $n\geq 3$.

For completeness, let us state what happens in the rather different case $n=2$.
\begin{Proposition} \label{density2}
  Let $n=2$. Then $D_n$ is a multiplicative function, and $D_n(p^m)$ is strictly decreasing as a function of $m$ for any prime $p$. We have
    \[D_2(k_i)\to 0\]
  if and only if $\lim_{i\to\infty}\sum_{p|k_i}1/p\to\infty$. Moreover,
    \[D_2(k_i)\to 1\]
  if and only if $\lim_{i\to\infty}\sum_{p|k_i}1/p\to0$. The sums are taken over all primes $p$ which divide $k_i$.
\end{Proposition}

In light of Remark \ref{partial_limits_remark}, one may ask which values in the interval $[D_n(0),1]$ can be obtained as partial limits of the function $D_n$. In this direction, we have the following result.

\begin{Proposition} \label{densitydensity}
   For $n\geq 4$, the set of values of $D_n(k)$ as $k$ ranges over $\Z$ is not dense in in the interval $[D_n(0),1]$. For $n=2$, the set of values of $D_2(k)$ as $k$ ranges over $\Z$ is dense in the interval $[0,1]$.
\end{Proposition}

Section \ref{chapter_density} is dedicated to the proofs of Theorem \ref{thm_density}, Proposition \ref{density2} and Proposition \ref{densitydensity}.

\subsection{Proof outline of Theorem \ref{main_theorem}} \label{section_proof_ideas}
Our proof of Theorem \ref{main_theorem} uses essentially the same approach as \cite{DRS}.
The set $M'_{n,k}$ is partitioned into a finite number of orbits $A\SL_n(\Z)$, where $A\in M_{n,k}$ are matrices in Hermite normal form with primitive row vectors. We count the matrices in each orbit separately. The number of matrices in each orbit scales as a fraction $1/k^{n-1}$ of the number of matrices in $\SL_n(\Z)$. We can view $\SL_n(\Z)$ as a lattice in the space $\SL_n(\R)$, and the problem is reduced to a lattice point counting problem. The lattice points inside the ball $B_T$ are counted by evaluating the normalized Haar measure of $B_T\cap\SL_n(\R)$. 

\section{Preliminaries}

The \dfn{Riemann zeta function} $\zeta$\index{$\zeta$, Riemann zeta function} is given by
\[\zeta(s)\defeq \sum_{n=1}^\infty \dfrac1{n^s}=\prod_p\dfrac1{1-1/p^s}\]
for $\Re s>1$, where we use the convention that when an index $p$ is used in a sum or product, it ranges over the set of primes.

The \dfn{Möbius function}\index{$\mu$, Möbius function} $\mu$ is defined by
$\mu(k)\defeq (-1)^m$ if $k$ is a product of $m$ distinct prime factors (that is, $k$ is \dfn{square-free}),\index{square-free} and $\mu(k)\defeq 0$ otherwise. We note that $\mu$ is a \dfn{multiplicative function}, that is, a function ${f:\N^*\to\C}$ defined on the positive integers such that $f(ab)=f(a)f(b)$ for all coprime $a,b$. 

We will use the fact that $\SL_n(\R)=M_{n,1}$ has a normalized Haar measure $\nu$ which is bi-invariant (see \cite{MeanValueTheorem}).

\subsection{Lattice point counting} \label{section_lattice_point_counting}
Let $G$ be a topological group with a normalized Haar measure $\nu_G$ and a lattice $\Gamma\subseteq G$, and let $G_T$ be an increasing family of bounded subsets of $G$ for all $T\geq 1$. Under certain conditions (see for instance \cite{GorodnikNevo}), we have
\[\abs{G_T\cap\Gamma} \sim \nu_G(G_T\cap G),\]
where we by $f(T)\sim g(T)$ mean that $f(T)/g(T)\to 1$ as $T\to\infty$. 
In this paper, we are interested in the lattice $\SL_n(\Z)$ inside $\SL_n(\R)$, and the following result will be crucial.
\begin{Theorem}[\!\!\cite{DRS}, Theorem 1.10] \label{drs_lemma}
Let $B_T$ be the ball of radius $T$ in the space $M_n(\R)$ of real $n\times n$-matrices under the Euclidean norm $\|A\|=\sqrt{\tr(A\T A)}$. Let $\nu$ be the normalized Haar measure of $\SL_n(\R)$. Then
\begin{gather}
  \abs{B_T\cap \SL_n(\Z)} = \nu(B_T\cap \SL_n(\R)) +\BigO_\e(T^{n(n-1)-1/(n+1)+\e})
\end{gather}
for all $\e>0$, and the main term is given by
\[|B_T\cap \SL_n(\Z)| \sim C_1T^{n(n-1)}, \quad C_1=\dfrac1{\zeta(2)\cdots\zeta(n)}\dfrac{\pi^{n^2/2}}{\Gamma\left(\dfrac n2\right)\Gamma\left(\dfrac{n(n-1)}2+1\right)}.\]
\end{Theorem}

In fact, a slightly more general statement is true. We can replace the balls $B_T$ in Theorem \ref{drs_lemma} with balls under any norm on $M_n(\R)$, and the asymptotics will still hold, save for a slighty worse exponent in the error term.

\begin{Theorem}[\!\!\cite{GorodnikNevo}, Corollary 2.3] \label{gn_cor}
Let $\|\cdot\|'$ be any norm on the vector space $M_n(\R)$, and let $G_T$ be the ball of radius $T$ in $M_n(\R)$ under this norm. Let $\nu$ be the normalized Haar measure of $\SL_n(\R)$. Then
\[|G_T\cap \SL_n(\Z)| = \nu(G_T\cap \SL_n(\R)) + \BigO_\e(T^{n(n-1)-1/(2n)+\e})\]
for all $\e>0$.
\end{Theorem}

We will be interested in the following particular case of Theorem \ref{gn_cor}. Let $A\in M_{n,k}$. Then $\|X\|'\defeq \|A\inv X\|$ defines a norm on $M_n(\R)$, and the ball of radius $T$ in $M_n(\R)$ under the norm $\|\cdot\|'$ is $A\cdot B_T$.

\begin{Corollary} \label{drs_cor}
Let $A\in M_{n,k}$. Then
\[|AB_T\cap \SL_n(\Z)|=\nu(AB_T\cap \SL_n(\R))+\minfelterm\]
for all $\e>0$, using the notation from Theorem \ref{drs_lemma}.
\end{Corollary}

\section {The number of matrices with primitive rows \label{counting_problem}}
In the present section, we will prove Theorem \ref{main_theorem}. 
We begin by noting that the common divisors of the entries of each row in an integer $n\times n$-matrix $A$ are preserved under multiplication on the right by any matrix $X\in\SL_n(\Z)$. 
In particular, if each row of $A$ is primitive, then each row of $AX$ is primitive, for any $X\in\SL_n(\Z)$. So we get:

\begin{Lemma}
If $A\in M'_{n,k}$ then $AX\in M'_{n,k}$ for all $X\in\SL_n(\Z)$. Thus $A\cdot \SL_n(\Z)\subseteq M'_{n,k}$.
\end{Lemma}

Consequently $M'_{n,k}$ may be written as a disjoint union of orbits of $\SL_n(\Z)$: \[M'_{n,k}=\bigcup_{A\in \mathcal A}A\SL_n(\Z),\]
for properly chosen subsets $\mathcal A$ of $M'_{n,k}$. In fact, as we will show in the following, the number of orbits is finite, and so we may take $\mathcal A$ to be finite.

A lower triangular integer matrix
  \[C\defeq \mat{
  c_{11}&0&\cdots&0\\
  c_{21}&c_{22}&\ddots&0 \\
  \vdots&&\ddots&0\\
  c_{n1}&\cdots&c_{n(n-1)}&c_{nn}
  }
  \]
is said to be in (lower) \dfn{Hermite normal form} if $0<c_{11}$ and $0\leq c_{ij}<c_{ii}$ for all $j<i$.
The following result is well-known.

\begin{Lemma}[\!\!\cite{Cohen}, Theorem 2.4.3] \label{lemma_hnf}
Assume $k>0$. Given an arbitrary matrix $A\in M_{n,k}$, the orbit $A\SL_n(\Z)$ contains a unique matrix  $C$ in Hermite normal form.
\end{Lemma}

We may thus write
\[M'_{n,k}=\bigcup_{i=1}^m A_i\SL_n(\Z),\]
where $A_1,\ldots,A_m$ are the unique matrices in Hermite normal form with primitive row vectors and determinant $k$, and $m\defeq|M'_{n,k}/\SL_n(\Z)|$. By counting the number of matrices in Hermite normal form with determinant $k>0$, we get

\begin{gather}
  |M_{n,k}/\SL_n(\Z)|=\sum_{d_1\cdots d_n=k} d_1^0d_2^1\cdots d_n^{n-1},  \label{formula_an}
\end{gather}
where the sum ranges over all positive integer tuples $(d_1,\ldots,d_n)$ such that $d_1\cdots d_n=~k$.

\begin{Proposition} \label{formula_anprime_hehe}
Let $k>0$. Then
\[|M'_{n,k}/\SL_n(\Z)|=
\sum_{d_1\cdots d_n=k}\prod_{i=1}^n \sum_{g|d_i}\mu(g)\left(\dfrac {d_i}g\right)^{i-1}
\]
where the first sum ranges over all positive integer tuples $(d_1,\ldots,d_n)$ such that $d_1\cdots d_n=k$.
\end{Proposition}


\begin{proof}
We want to count those matrices in Hermite normal form which are in $M'_{n,k}$, that is, $n\times n$-matrices in Hermite normal form with determinant $k$ and all rows primitive. The number of such matrices is
\[\left|M'_{n,k}/\SL_n(\Z)\right|=\sum_{d_1\cdots d_n=k}\prod_{i=1}^n v_{i}(d_i),\]
where $v_{i}(d)$ is the number of primitive vectors $(x_1,\ldots,x_{i-1},d)$ such that $0\leq x_1,\ldots,x_{i-1}<d$.
There is a bijective correspondence between the primitive vectors $(x_1,\ldots,x_{i-1},d)$ and the vectors $y=(y_1,\ldots,y_{i-1})$ such that $1\leq y_1,\ldots,y_{i-1}\leq d$ and $\gcd(y)$ is coprime to $d$. Let $d=p_1^{a_1}\cdots p_j^{a_j}$ be the prime factorization of $d$. The number of vectors $y$ which are divisible by some set of primes $P\subseteq\{p_1,\ldots,p_j\}$ is
\[\left(\dfrac{d}{\prod_{p\in P} p}\right)^{i-1},\]
so by the principle of inclusion/exclusion (see \cite{Stanley}), we have
\begin{gather}
v_{i}(d)=\sum_{P\subseteq\{p_1,\ldots,p_j\}}(-1)^{|P|}\left(\dfrac{d}{\prod_{p\in P} p}\right)^{i-1}\\
=\sum_{g|p_1\cdots p_j}\mu(g)\left(\dfrac dg\right)^{i-1}=\sum_{g|d}\mu(g)\left(\dfrac dg\right)^{i-1}. 
\end{gather}
\end{proof}

We are now ready to derive the asymptotics of $N'_{n,k}(T)$.

\begin{proof}[Proof of Theorem \ref{main_theorem}]
  Let us write $A_1,\ldots,A_m$ for all the $n\times n$-matrices in Hermite normal form with determinant $k$, where $m\defeq |M'_{n,k}/\SL_n(\Z)|,$ and let $1\leq i\leq m$. Then
  \begin{gather}
    \abs{B_T\cap A_i\SL_n(\Z)}=\abs{A_i(A_i\inv B_T\cap \SL_n(\Z))}=\abs{A_i\inv B_T\cap \SL_n(\Z)},
  \end{gather}
  which by Corollary \ref{drs_cor} is equal to
  \begin{gather}
    \nu(A_i\inv B_T\cap \SL_n(\R))+\minfelterm
  \end{gather}
  for any $\e>0$. Since $A_i/k^{1/n}\in\SL_n(\R)$, we get by the invariance of the measure $\nu$ that
  \begin{gather}
    \nu(A_i\inv B_T\cap \SL_n(\R)) =
    \nu\left(\dfrac{A_i}{k^{1/n}} \left(A_i\inv B_T\cap \SL_n(\R)\right)\right) = \\
    \nu\left(k^{-1/n} B_T\cap \dfrac{A_i}{k^{1/n}}\SL_n(\R)\right) = 
    \nu\left(B_{T/k^{1/n}}\cap \SL_n(\R)\right).
  \end{gather}
  By Theorem \ref{drs_lemma}, the last expression is equal to
  \begin{gather}
    C_1(T/k^{1/n})^{n(n-1)}+\minfelterm,
  \end{gather}
  and thus
  \begin{gather}
    \abs{B_T\cap A_i\SL_n(\Z)} = \dfrac{C_1}{k^{{n-1}}}T^{n(n-1)}+\minfelterm. \label{hello}
  \end{gather}

  Now,
  \begin{gather}
    N_{n,k}'(T)=\abs{B_T\cap M_{n,k}'}=\abs{B_T\cap \bigcup_{i=1}^m A_i\SL_n(\Z)}
    =\sum_{i=1}^m\abs{B_T\cap A_i\SL_n(\Z)},
  \end{gather}
  so applying \eqref{hello} we get
  \begin{gather}
    N_{n,k}'(T)=\sum_{i=1}^m\dfrac{C_1}{k^{{n-1} }}T^{n(n-1)}+\minfelterm = \\
    \abs{M'_{n,k}/\SL_n(\Z)} \dfrac{C_1}{k^{{n-1} }}T^{n(n-1)}  +\minfelterm,
  \end{gather}
  and we need only apply Proposition \ref{formula_anprime_hehe} to get an explicit constant for the main term. This concludes the proof.
\end{proof}

\section{Density of matrices with primitive rows}
\label{chapter_density}

\newcommand{\all}{|M_{n,k}/\SL_n(\Z)|}
\newcommand{\prim}{|M'_{n,k}/\SL_n(\Z)|}

Set
\begin{align}
a_n(k)&\defeq \all=\sum_{d_1\cdots d_n=k}d_1^0\cdots d_n^{n-1}, \label{def_an} \\
a'_n(k)&\defeq \prim=\sum_{d_1\cdots d_n=k}\prod_{i=1}^n\sum_{g|d_i}\mu(g)\left(\dfrac{d_i}g\right)^{i-1}. \label{def_anprime}
\end{align}

We would like to calculate the density of matrices with primitive rows in $M_{n,k}$ for $k\neq 0$, that is, the quantity
\[D_n(k)= \lim_{T\to\infty}\dfrac{N'_{n,k}(T)}{N_{n,k}(T)}=\dfrac{c'_{n,k}}{c_{n,k}}=\dfrac{|M'_{n,k}/\SL_n(\Z)|}{|M_{n,k}/\SL_n(\Z)|}=\dfrac{a'_n(k)}{a_n(k)}.\]
We will prove in section \ref{proofconjsection} that $a_n, a'_n$ and $D_n$ are multiplicative functions, and therefore we need only understand their behavior for prime powers $k=p^m$. We will now prove a sequence of lemmas which we will finally use in section \ref{density_asymptotics} to prove Theorem \ref{thm_density}.

\begin{Lemma} \label{incl_excl_an}
  The functions $a'_n$ and $a_n$ are connected via the identity
    \[a'_n(p^m)=\sum_{i=0}^m(-1)^i\binom ni a_n(p^{m-i})\]
  for primes $p$ and $m\geq 0$.
\end{Lemma}
\begin{proof}
  $a_n(p^m)$ counts the number of $n\times n$-matrices in Hermite normal form with determinant $p^m$, whereas $a'_n(p^m)$ counts the number of such with primitive rows.
  If $A$ is a matrix in $M_{n,k}\setminus M'_{n,k}$, then some set of rows, indexed by $S\subseteq[n]\defeq\{1,\ldots,n\}$ (where $|S|\leq m$), are divisible by $p$. The number of such matrices is $a_n(p^{m-|S|})$, and thus by the inclusion/exclusion principle,
  \[a'_n(p^m)=\sum_{\substack{S\subseteq[n]\\|S|\leq m}}(-1)^{|S|}a_n(p^{m-|S|})=\sum_{i=0}^m(-1)^i\binom ni a_n(p^{m-i}). \qedhere \]
\end{proof}

\begin{Lemma} \label{split_an}
  For any prime $p$ and $m\geq 1$, the following recursion holds:
  \[a_n(p^m)=p^{n-1}a_n(p^{m-1})+a_{n-1}(p^m),\]
  or equivalently,
  \[a_n(p^{m-1})=\dfrac{a_n(p^m)-a_{n-1}(p^m)}{p^{n-1}}.\]
\end{Lemma}
\begin{proof}
  We split the sum
  \[a_n(p^m)=\sum_{d_1\cdots d_n=p^m}d_1^0\cdots d_n^{n-1}\]
  into two parts, one part where $d_n$ is divisible by $p$, and another part where it is not (so that $d_n=1$). The terms corresponding to $d_n=1$ sum to $a_{n-1}(p^m)$. Where $d_n$ is divisible by $p$, we can write $d_n\eqdef pe_n$ for some $e_n$. Let $e_i\defeq d_i$ for all $i<n$. Thus,
    \begin{gather}
      \sum_{\substack{d_1\cdots d_n=p^m\\p|d_n}}d_1^0\cdots d_n^{n-1}=
      \sum_{\substack{e_1\cdots e_n=p^{m-1}}}e_1^0\cdots (pe_n)^{n-1}=p^{n-1}a_n(p^{m-1}).
    \end{gather}
  Adding the two parts gives us $a_n(p^m)=p^{n-1}a_n(p^{m-1})+a_{n-1}(p^m)$, from which the second claim in the lemma follows by rearrangement.
\end{proof}

\begin{Lemma} \label{limit_pm}
  Let $n$ and $p$ be fixed, where $n\geq 3$ and $p$ is a prime. Then
  \[D_n(p^m)\to \left(1-\dfrac1{p^{n-1}}\right)^n\]
  as $m\to\infty$.
\end{Lemma}
\begin{proof}
  We apply the simple upper bound
  \begin{gather}
    a_{n-1}(p^m)=\sum_{d_1\cdots d_{n-1}=p^m}d_1^0\cdots d_n^{n-2}\leq \sum_{d_1\cdots d_{n-1}=p^m}(p^m)^{n-2}=(m+1)^{n-1}(p^m)^{n-2}
  \end{gather}
  to the expression for $a_n(p^{m-1})$ in Lemma \ref{split_an}:
  \begin{align}
    a_n(p^{m-1})&=\dfrac1{p^{n-1}}(a_n(p^m)-a_{n-1}(p^m))\\
    &=\dfrac1{p^{n-1}}a_n(p^m)+\BigO((p^m)^{n-2}(m+1)^{n-1}).
  \end{align}
  Repeated application (at most $n$ times) of this formula yields the asymptotics
  \[a_n(p^{m-i})=\dfrac1{(p^{n-1})^i}a_n(p^m)+\BigO((p^m)^{n-2}(m+1)^{n-1})\]
  for $1\leq i\leq n$.

  Now let $m\to\infty$, so that we may assume $m$ to be larger than $n$. The sum in Lemma \ref{incl_excl_an} then extends up to $i=n$ (because the factors $\binom ni$ vanish for larger $i$), so
  \begin{align}
    a'_n(p^m)&=\sum_{i=0}^n(-1)^i\binom ni a_n(p^{m-i})\\
      &=\sum_{i=0}^n(-1)^i\binom ni \dfrac1{(p^{n-1})^i}a_n(p^m)+\BigO((p^m)^{n-2}(m+1)^{n-1}).
  \end{align}
  We divide by $a_n(p^m)$ on both sides and use the fact that $a_n(p^m)\geq (p^m)^{n-1}$, so that
  \begin{align}
    D_n(p^m)&=\sum_{i=0}^n(-1)^i\binom ni \dfrac1{(p^{n-1})^i}+\BigO\left(\dfrac{(p^m)^{n-2}(m+1)^{n-1}}{(p^m)^{n-1}}\right)\\
    &=\sum_{i=0}^n\binom ni \left(\dfrac{-1}{p^{n-1}}\right)^i+\BigO\left(\dfrac{(m+1)^{n-1}}{p^m}\right)\\
    &=\left(1-\dfrac1{p^{n-1}}\right)^n+\BigO\left(\dfrac{(m+1)^{n-1}}{p^m}\right).
  \end{align}
  As $m\to\infty$, the second term on the right vanishes.
\end{proof}

\subsection{Multiplicativity and monotonicity of the density function} \label{proofconjsection}

In this section we will prove the following proposition.
\begin{Proposition} \label{decreasing}
  The function $D_n$ is multiplicative, and $D_n(p^m)$ is strictly decreasing as a function of $m$ for any fixed prime $p$ and dimension $n\geq 2$.
\end{Proposition}

We may rewrite \eqref{def_an} as
\begin{gather}
  a_n=(\cdot)^{n-1}*\cdots*(\cdot)^0
\end{gather}
where $(\cdot)^i$ is the function $x\mapsto x^i$ and $*$ denotes the Dirichlet convolution. Similarly, we may rewrite \eqref{def_anprime} as
\begin{gather}
  a_n' = (\mu*(\cdot)^{n-1}) * \cdots * (\mu*(\cdot)^0), \label{anprimemuu}
\end{gather}
so by the commutativity and associativity of the Dirichlet convolution we have
\begin{gather}
  a_n' = \mu^{*n}*a_n,
\end{gather}
where $\mu^{*n}$ denotes the convolution of $\mu$ with itself $n$ times (so that $\mu^{*1}=\mu$). Since the Dirichlet inverse of $\mu$ is the constant function $1$, we have also the relation
\begin{gather}
  a_n = 1^{*n}*a_n'.
\end{gather}
As $\mu$ and $(\cdot)^i$ are multiplicative functions, it follows that $a_n, a_n'$ and $D_n$ are multiplicative as well.

Now, we want to show that $D_n(p^m) = a_n'(p^m)/a_n(p^m)$ is strictly decreasing as a function of $m$, for fixed $n\geq 2$ and primes $p$, or equivalently that
\begin{gather}
  \dfrac{a_n'(p^m)}{a_n(p^m)} > \dfrac{a_n'(p^{m+1})}{a_n(p^{m+1})} \label{lul}
\end{gather}
for all $m\geq 0$.
The inequality \eqref{lul} is equivalent to
\begin{gather}
  \dfrac{a_n'(p^m)}{(1^{*n}*a_n')(p^m)} > \dfrac{a_n'(p^{m+1})}{(1^{*n}*a_n')(p^{m+1})}
\end{gather}
for all $m\geq 0$, which is equivalent to
\begin{gather}
  \dfrac{a_n'(p^m)}{\sum_{i=0}^m1^{*n}(p^i)a_n'(p^{m-i})} > \dfrac{a_n'(p^{m+1})}{\sum_{i=0}^{m+1}1^{*n}(p^i)a_n'(p^{m+1-i})},
\end{gather}
or, after taking the reciprocal of both sides,
\begin{gather}
  \sum_{i=0}^m1^{*n}(p^i)\dfrac{a_n'(p^{m-i})}{a_n'(p^m)} < \sum_{i=0}^{m+1}1^{*n}(p^i)\dfrac{a_n'(p^{m+1-i})}{a_n'(p^{m+1})}. \label{tricksy}
\end{gather}

Since the last term ($i=m+1$) on the right hand side is positive, this inequality holds if
\begin{gather}
  \dfrac{a_n'(p^{m-i})}{a_n'(p^m)} \leq \dfrac{a_n'(p^{m+1-i})}{a_n'(p^{m+1})}
\end{gather}
for all $i\leq m$. We can rearrange this inequality as
\begin{gather}
  \dfrac{a_n'(p^{m+1})}{a_n'(p^m)} \leq \dfrac{a_n'(p^{m+1-i})}{a_n'(p^{m-i})},
\end{gather}
which states that
$
  a_n'(p^{m+1})/a_n'(p^m)
$
is a non-increasing function of $m$, for fixed $n\geq 2$ and $p$ prime. We will therefore be done if we can prove that
\begin{gather}
    a_n'(p^{m+1}) a_n'(p^{m+1}) \geq a_n'(p^m)a_n'(p^{m+2}) \label{anprimelogconcave}
\end{gather}
for all $m\geq 0$, or equivalently, that the function $m\mapsto a_n'(p^m)$ is logarithmically concave:

We say that a sequence $u:\N_0\to\R$ is \dfn{logarithmically concave} if
\begin{gather}
    u_r^2-u_{r-1}u_{r+1} \geq 0 \label{logkdef}
\end{gather}
for all $r\geq 1$. We note that a sequence $u$ of positive real numbers is logarithmically concave if and only if $u_1/u_0 \geq u_2/u_1 \geq u_3/u_2 \geq \cdots$, that is, if and only if $(u_1/u_0, u_2/u_1, u_3/u_2, \ldots)$ is a non-increasing sequence. Also note that if $u$ is positive and logarithmically concave, then the inequality $u_{i+1}/u_i \geq u_{j+1}/u_j$ implies the inequality $u_{i+1}u_j-u_{j+1}u_i\geq 0$ for all indices $i<j$.

Let $\star$ denote the discrete convolution, so that $(u\star v)_r=\sum_{j=0}^r u_{r-j}v_j$ for all $r\geq 0$ given any sequences $u,v:\N_0\to\R$. We will need the following fact, which follows from the proof of Theorem 1 in \cite{menon}.

\begin{Theorem}[\!\!\cite{menon}, Theorem 1] \label{menonlemma} 
    Let $u,v:\N_0\to\R$ be sequences such that $u_0=v_0=1$, and let $w=u\star v$. Then we may write
    $w_r^2-w_{r-1}w_{r+1}=\op{I}+\op{II}+\op{III},$  where
    \begin{align}
        \op{I}&=\sum_{0\leq i< j\leq r-1} (v_j v_{i+1} - v_{j+1}v_i) ( u_{r-j}  u_{r-i-1} -  u_{r-1-j} u_{r-i}),\\
        \op{II}&=\sum_{j=0}^{r-1} v_j ( u_{r-j}  u_{r} -  u_{r-1-j} u_{r+1}),\\
        \op{III}&=v_ru_{r} + \sum_{j=0}^{r-1} u_{j} (v_r v_{r-j}-v_{r+1} v_{r-1-j}),
    \end{align}
    for all $r\geq 1$. In particular, if $u,v$ are positive and logarithmically concave sequences, then so is $w$, since all factors in the sums in $\op{I},\op{II},\op{III}$ are non-negative for such $u,v$.
\end{Theorem}

Fix $n$ and $p$. Then since $(\mu*(\cdot)^{i})(p^m)=\sum_{r=0}^m\mu(p^{m-r})p^{ri}=(M\star P_i)(m)$ where $M$ is the sequence $(1,-1,0,0,0,\ldots)$ and where $P_i$ is the sequence $(1,p^i,p^{2i},p^{3i},\ldots)$, equation \eqref{anprimemuu} implies that the function $m\mapsto a_n'(p^m)$ can be written as
\begin{gather}
    (M\star P_{n-1})\star\cdots\star (M\star P_{0}).
\end{gather}

\begin{Lemma} \label{convolemma}
    Let $0\leq i<j$. Then
    $(M\star P_i)\star(M\star P_j)$ is positive and logarithmically concave if and only if $i>0$.
\end{Lemma}
\begin{proof}
    Write $u\defeq M\star P_i$ and $v\defeq M\star P_j$ where $i<j$. We have $u_0=1$ and $u_r=p^{ir}-p^{i(r-1)}$ for all $r\geq 1$. Thus $u_1u_r-u_0u_{r+1}=(p^i-1)(p^{ir}-p^{i(r-1)})-(p^{i(r+1)}-p^{ir})=p^{i(r-1)}-p^{ir} = -u_r$ for all $r\geq 1$, and $u_su_r-u_{s-1}u_{r+1}=0$ when $s\geq 2, r\geq 1$ or $s=1,r=0$. Likewise $v_sv_r-v_{s-1}v_{r+1}$ is $-v_r$ if $s=1, r\geq 1$, and $0$ otherwise.

    Let $w\defeq u\star v=(M\star P_i)\star(M\star P_j)$.
    By Theorem \ref{menonlemma} we can write $w_r^2-w_{r-1}w_{r+1}=\op{I}+\op{II}+\op{III}$, where
    \begin{align}
        \op{I}&=(-v_{r-1})(-u_{r-1}),\\
        \op{II}&= v_{r-1} (-u_r),\\
        \op{III}&=v_ru_{r} +  u_{r-1} (-v_r),
    \end{align}
    for all $r\geq 1$, and therefore
    \begin{gather}
        w_r^2-w_{r-1}w_{r+1} =
         u_{r-1} v_{r-1}  +u_{r}v_r - u_r v_{r-1}   -  u_{r-1} v_r =\\ (u_r-u_{r-1})(v_r-v_{r-1}).
    \end{gather}
    Thus, since $(u_0,u_1,\ldots)$ is a non-decreasing sequence for $i>0$, and likewise $(v_0,v_1,\ldots)$ is a non-decreasing sequence for $j>0$, we get $w_r^2-w_{r-1}w_{r+1}\geq 0$ for all $r\geq 1$ for $i>0$. Also, the sequence $w$ is positive for $i,j>0$ since it is then the convolution of two positive sequences.
     If $i=0$, then the inequality $w_r^2-w_{r-1}w_{r+1}\geq 0$ fails for $r=1$ since then $u_1-u_0=(p^0-1)-1<0$ and $v_1-v_0=(p^j-1)-1>0$.
\end{proof}

We will prove Proposition \ref{decreasing} by induction on $n$. The base case is the following proposition, which we will prove in Appendix \ref{appxfourandfive}.
\begin{Proposition} \label{fourandfive}
    For $n=4,5$ and any fixed prime $p$, the function $m\mapsto a_n'(p^m)$ is logarithmically concave.
\end{Proposition}
It happens that $a_n'(p^m)$, as a function of $m$, is not logarithmically concave for $n=2$ or $n=3$ for all $p$ (it fails the inequality \eqref{anprimelogconcave} for $r=1$ when $p=2$), so we will also need the following proposition, which we prove in Appendix \ref{appxtwoandthree}.

\begin{Proposition} \label{twoandthree}
    For $n=2,3$ and any fixed prime $p$, the function $m\mapsto D_n(p^m)$ is strictly decreasing.
\end{Proposition}

The proofs of Propositions \ref{twoandthree} and \ref{fourandfive} consist of explicitly evaluating $a_n(p^m)$ and $a_n'(p^m)$ for the values of $n$ in question, both of which are polynomials in $p$ with exponents in $m$, and verifying equations \eqref{lul} and \eqref{anprimelogconcave}, respectively.

\begin{proof}[Proof of Proposition \ref{decreasing}]
    By Proposition \ref{fourandfive} and Proposition \ref{twoandthree}, it suffices to consider $n>5$.
    By Proposition \ref{fourandfive} and Theorem \ref{menonlemma}, it follows that $A_n'(m)\defeq a_n'(p^m)$ is logarithmically concave for all $n>5$ and any $p$, since for any even $n>5$, we can write
\begin{gather}
    A_n'=A_4'\star[(M\star P_4)\star(M\star P_5)]\star\cdots\star[(M\star P_{n-2})\star(M\star P_{n-1})],
\end{gather}
and for any odd $n>5$, we can write
\begin{gather}
    A_n'=A_5'\star[(M\star P_5)\star(M\star P_6)]\star\cdots\star[(M\star P_{n-2})\star(M\star P_{n-1})],
\end{gather}
and in both cases we have written $A_n'$ as the convolution of positive and logarithmically concave sequences, by Lemma \ref{convolemma}. We have thus proven the inequality \eqref{anprimelogconcave}, and this concludes the proof of Proposition \ref{decreasing}.
\end{proof}

\subsection{Asymptotics of the density function}\label{density_asymptotics}
In this section we prove Theorem \ref{thm_density} and thus derive the asymptotics of $D_n(k)$. Fix $n\geq 3$.  For any nonzero integer $k_i$, write $k_i=\prod_p p^{m_p(i)}$ as a product of prime powers, where all but finitely many of the exponents $m_p(i)$ are zero. Then since $D_n$ is multiplicative, we have
\begin{gather}
  D_n(k_i)=\prod_p D_n(p^{m_p(i)}).
\end{gather}
Now, by Lemma \ref{limit_pm} and Proposition \ref{decreasing}, we get
\begin{gather}
  1\geq \prod_p D_n(p^{m_p(i)}) > \prod_p\left(1-\dfrac1{p^{n-1}}\right)^n = \dfrac{1}{\zeta(n-1)^n}>0,
\end{gather}
so it follows by dominated convergence that 
\begin{gather}
  \lim_{i\to\infty} \prod_p D_n(p^{m_p(i)}) = \prod_p \lim_{i\to\infty} D_n(p^{m_p(i)}). \label{ili}
\end{gather}
whenever $(k_1,k_2,\ldots)$ is a sequence of nonzero integers such that the limit $\lim_{i\to\infty} D_n(p^{m_p(i)})$ exists for each prime $p$.

Let $(k_1,k_2,\ldots)$ be a sequence of nonzero integers. It now follows from \eqref{ili}, Proposition \ref{decreasing} and the fact that $D_n(1)=1$, that \[D_n(k_i)\to 1\] if and only if $m_p(i)\to 0$ as $i\to\infty$ for all $p$, that is, if and only if $(k_1,k_2,\ldots)$ is a rough sequence. Likewise it follows, using Lemma \ref{limit_pm}, that \[D_n(k_i)\to \dfrac1{\zeta(n-1)^n}\] if and only if $m_p(i)\to \infty$ for all $p$, that is, if and only if $(k_1,k_2,\ldots)$ is a totally divisible sequence. Since $D_n(0)=1/\zeta(n-1)^n$, we may allow the elements of the sequence $(k_1,k_2,\ldots)$ to also assume the value $0$.

Finally, it follows that $D_n(k)\to 1$ as $n\to\infty$ uniformly with respect to $k$ since
\begin{gather}
  D_n(k)\geq \dfrac1{\zeta(n-1)^n}\to 1
\end{gather}
as $n\to\infty$ because $\zeta(n-1)=1+O(2^{-n})$ for $n\geq 3$. We have thus proved all parts of Theorem \ref{thm_density}.\qed

We conclude this section by proving Proposition \ref{density2}, which tells us the asymptotics of $D_2(k)$ for $n=2$.
\begin{proof}[Proof of Proposition \ref{density2}]
  If $m=0$, we have $D_2(p^m)=1$. Assume $m>0$.
  The $2\times 2$-matrices in Hermite normal form with determinant $p^m$ and primitive rows are of the form $\mat{1&0\\x&p^m}$ where $0\leq x<p^m, p\nmid x$. Thus $a'_2(p^m)=p^m(1-1/p)$. Moreover,
  \[a_2(p^m)=\sum_{d_1d_2=p^m}d_2=\sum_{i+j=m}p^i=\sum_{i=0}^mp^i=\dfrac{p^{m+1}-1}{p-1}=p^m \dfrac{1-1/p^{m+1}}{1-1/p} ,\]
  so
  \begin{gather}
      D_2(p^m)=\dfrac{(1-1/p)^2}{1-1/p^{m+1}} \label{dtwo}
  \end{gather}
  for all $m\geq 1$.
  We see immediately that $D_2(p^m)$ is strictly decreasing as a function of $m$, for any fixed $p$. Therefore
  \[\left(1-\dfrac 1p\right)^2\leq D_2(p^m)\leq 1-\dfrac 1p.\]
  Since $D_2$ is multiplicative, we get
  \[\left[\prod_{p|k}\left(1-\dfrac 1p\right)\right]^2\leq D_2(k)\leq \prod_{p|k}\left(1-\dfrac 1p\right).\]
  The left and right sides both tend to $0$ if and only if $\lim_{i\to\infty}\sum_{p|k_i}1/p\to\infty$, and they both converge to $1$ if and only if $\lim_{i\to\infty}\sum_{p|k_i}1/p\to0$.
\end{proof}

\subsection{The image of the density function} \label{proof_of_image}

\begin{proof}[Proof of Proposition \ref{densitydensity} for $n\geq 4$]
    By Proposition \ref{decreasing}, the function $D_n$ is multiplicative, and $D_n(p^m)$ is strictly decreasing as a function of $m$ for any fixed $p,n$. Thus we get $D_n(k)\leq D_n(2)$ whenever $k$ is divisible by $2$. When $k$ is not divisible by $2$, we get
    \begin{gather}
        D_n(k)\geq \prod_{p\geq 3}\lim_{m\to\infty}D_n(p^m) = 
        \prod_{p\geq 3}\left(1-1/p^{n-1}\right)^n = \\
        \dfrac1{\left(1-1/2^{n-1}\right)^n}\prod_{p}\left(1-1/p^{n-1}\right)^n = 
        \dfrac1{\left(1-1/2^{n-1}\right)^n}\dfrac1{\zeta(n-1)^n}.
    \end{gather}
    by Lemma \ref{limit_pm}.
    We will show that this value is larger than $D_n(2)$, which will prove that the image of $D_n:\Z\to\R$ is not dense in $[D_n(0),1]$. By equation \eqref{def_an} we have $a_n(2)=\sum_{i=1}^n 2^{i-1}=2^n-1$ and by Lemma \ref{incl_excl_an} we have $a_n'(2)=a_n(2)-na_{n}(1)=(2^n-1)-n$, so $D_n(2)=1-n/(2^n-1)$.

    Thus it suffices to prove
    \begin{gather}
        \dfrac1{\left(1-1/2^{n-1}\right)^n}\dfrac1{\zeta(n-1)^n} > 1-\dfrac n{2^n-1}. \label{sweet}
    \end{gather}
    This inequality can be verified numerically for $n=4,5$. Let us now assume $n\geq 6$. The inequality \eqref{sweet} is is equivalent to
    \begin{gather}
        -\log(1-n/(2^n-1)) -n\log\left(1-1/2^{n-1}\right) > n\log\zeta(n-1).
    \end{gather}
    By Taylor expansion, the first term on the left hand side is $> n/(2^n-1)> n/2^n$, and the second term on the left hand side is $> n/2^{n-1}$. Thus the inequality above follows from 
    $
        1/2^n+1/2^{n-1} \geq \log\zeta(n-1),
    $
    or equivalently
    $
        e^{3/2^{n}} \geq \zeta(n-1).
    $
    We bound the left hand side from below by $1+3/2^{n}$, and we bound the right hand side from above by $1+1/2^{n-1}+\int_2^\infty \frac{\dd x}{x^{n-1}}$. Thus the inequality follows from
    $
        1+3/2^n \geq 1+1/2^{n-1}+1/((n-2)2^{n-2})
    $
    or equivalently
    $
        3 \geq 2+4/(n-2),
    $
    which is true for all $n\geq 6$.
\end{proof}

\begin{proof}[Proof of Proposition \ref{densitydensity} for $n=2$]
    It suffices to show that the set of values of $-\log(D_2(k))$ as $k$ ranges over positive square-free integers is dense in $[0,\infty)$. By the identity \eqref{dtwo} we have
    \begin{gather}
        D_2(p)=\dfrac{(1-1/p)^2}{1-1/p^2} = \dfrac{1-1/p}{1+1/p} = \dfrac{p-1}{p+1}=1-\dfrac2{p+1}.
    \end{gather}
    Let $k>0$ be squarefree, and let $P_0$ be the set of primes dividing $k$. Then
    \begin{gather}
        -\log D_2(k) = -\log\prod_{p\in P_0}D_2(p)=\sum_{p\in P_0}\left(-\log\left(1-\dfrac2{p+1}\right)\right).
    \end{gather}
    The terms $d_p\defeq -\log(1-\frac2{p+1})$ are positive, decreasing, and tend to zero as $p\to\infty$. By Taylor expansion, the sum $\sum_p d_p$ over all primes is larger than $\sum_{p}\frac2{p+1}$, which diverges since $\sum_p 1/p$ diverges.

    Now, given any $x\in[0,\infty)$ and any $\e>0$, we can find a $k$ such that $-\log D_2(k)$ is within a distance $\e$ from $x$ as follows. Let $p_0$ be the smallest prime such that $d_{p_0}<\e$, and let $P_0$ be the smallest set of consecutive primes, starting with $p_0$, such that $\sum_{p\in P_0}d_p\geq x$. Then the sum $\sum_{p\in P_0}d_p=-\log D_2(k)$ is at a distance at most $d_{p_0}<\e$ from $x$ since $d_p$ is decreasing, where $k=\prod_{p\in P_0}p$, and we are done.
\end{proof}

\appendix

\section{Proof of Proposition \ref{twoandthree}} \label{appxtwoandthree}
We prove Proposition \ref{twoandthree}. Recall from equation \eqref{dtwo} that $D_2(p^m)=(1-1/p)^2/(1-1/p^{m+1})$ if $m>0$, and otherwise $D_2(p^0)=1$. Thus we see immediately that $D_2(p^m)$ is strictly decreasing as a function of $m$.

The case $n=3$ remains. By equation \eqref{def_an} we get
\begin{gather}
    a_3(p^m)=
    \sum_{j_1+j_2+j_3=m}p^{j_2+2j_3}=
    \sum_{j_3=0}^m p^{2j_3}\sum_{j_2=0}^{m-j_3}p^{j_3}=
    \sum_{j_3=0}^m p^{2j_3}\dfrac{1-p^{m-j_3+1}}{1-p}=\\
    \dfrac1{1-p}\sum_{j_3=0}^m (p^{2j_3}-p^{m+j_3+1})=
    \dfrac1{1-p}\left(\dfrac{1-p^{2(m+1)}}{1-p^2}-p^{m+1}\dfrac{1-p^{m+1}}{1-p}\right)=\\
    \dfrac{1-p^{2(m+1)}-(1+p)p^{m+1}+(1+p)p^{2(m+1)}}{(1-p)(1-p^2)}=\\
    \dfrac{1-p^{m+1}-p^{m+2}+p^{(m+1)+(m+2)} }{(1-p)(1-p^2)}=
    \dfrac{(p^{m+1}-1)(p^{m+2}-1) }{(p-1)(p^2-1)}.
\end{gather}
for all $m\geq 1$. 
Let us write $I(P)\defeq 1$ if the condition $P$ is true, and $I(P)\defeq 0$ if the condition $P$ is false.
By equation \eqref{def_anprime} we get for all $m\geq 1$ that
\begin{gather}
    a_3'(p^m)=
    \sum_{\substack{ j_1,j_2,j_3\geq 0: \\ p^{j_1}p^{j_2}p^{j_3}=p^m }} \prod_{i=1}^3 \sum_{\substack{r\geq 0:\\ p^r|p^{j_i}}}\mu(p^r)(p^{j_i-r})^{i-1}=
    \\
    \sum_{\substack{j_2,j_3\geq 0:\\ j_2+j_3=m} }\left(p^{j_2}-p^{j_2-1}I(j_2>0)\right)\left(p^{2j_3}-p^{2(j_3-1)}I(j_3>0)\right). \label{aswedidforthree}
\end{gather}
We expand the product in the summand and split the sum into several geometric series which we sum individually. We get
\begin{gather}
    \sum_{j_2=0}^m\left(p^{2m-j_2}-p^{2m-j_2-1}I(j_2>0)-p^{2m-j_2-2}I(m>j_2)+p^{2m-j_2-3}I(0<j_2<m)\right)=\\
    p^{2m}\left(\dfrac{1-p^{-(m+1)}}{1-p^{-1}}-p\inv\left(\dfrac{1-p^{-(m+1)}}{1-p^{-1}}-1\right)-p^{-2}\dfrac{1-p^{-m}}{1-p^{-1}}+p^{-3}\left(\dfrac{1-p^{-m}}{1-p^{-1}}-1\right)\right)=\\
    \dfrac{p^{2m}}{1-p\inv} \Big(1-p^{-(m+1)}-p\inv\left(1-p^{-(m+1)}-(1-p^{-1})\right)-\\
        p^{-2}(1-p^{-m})+p^{-3}(1-p^{-m}-(1-p^{-1}))\Big)=\\
    \dfrac{p^{2m}}{1-p\inv} \left(
        (1-p^{-2})^2-p^{-m-1}(1-p\inv)^2
    \right) = 
    \dfrac{
        p^{2m}(1-p^{-2})^2
        -p^{m-1}(1-p\inv)^2
    }{1-p\inv} = \\
        (p^{2m}(p+1)^2 -p^{m+1})
    \dfrac{p-1}{p^3}.
\end{gather}
Since $D_3(1)=1$ and $D_3(p^m)<1$ for all $m>0$ (the diagonal matrix with diagonal entries $1,1,p^m$ is in Hermite normal form, but its last row is not primitive), it suffices to show that $D_3(p^m)>D_3(p^{m+1})$ for all $m\geq 1$. To see this, we note that
\begin{gather}
    \dfrac{p^{2m}(p+1)^2 -p^{m+1}}  {(p^{m+1}-1)(p^{m+2}-1)} > \dfrac{p^{2m+2}(p+1)^2 -p^{m+2}}  {(p^{m+2}-1)(p^{m+3}-1)} \iff \\
    p^{2m}(p+1)^2[(p^{m+3}-1)-p^2(p^{m+1}-1)] - p^{m+1}[(p^{m+3}-1)-p(p^{m+1}-1)] > 0 \iff \\
    p^{2m}(p+1)^2(p^2-1) - p^{m+1}(p^{m+2}+1)(p-1) > 0 \iff \\
    p^{2m}(p+1)^3 - p^{2m}(p^{3}+p^{1-m}) > 0,
\end{gather}
where the last inequality is true since $(p+1)^3>p^3+1\geq p^3+p^{1-m}$ for all $m\geq 1$ and all $p\geq 2$. This concludes the proof of Proposition \ref{twoandthree}.\qed

\section {Proof of Proposition \ref{fourandfive}} \label{appxfourandfive}
\subsection{The case $n=4$}
We prove Proposition \ref{fourandfive} for $n=4$. By equation \eqref{def_anprime}, we can write 

\begin{gather}
    a_4'(p^m)=
    \sum_{\substack{ j_1,j_2,j_3,j_4\geq 0: \\ p^{j_1}p^{j_2}p^{j_3}p^{j_4}=p^m }} \prod_{i=1}^4 \sum_{\substack{r\geq 0:\\ p^r|p^{j_i}}}\mu(p^r)(p^{j_i-r})^{i-1}=
    \\
    \sum_{j_2+j_3+j_4=m}\left(p^{j_2}-p^{j_2-1}I(j_2>0)\right)\left(p^{2j_3}-p^{2(j_3-1)}I(j_3>0)\right)\left(p^{3j_4}-p^{3(j_4-1)}I(j_4>0)\right),
\end{gather}
where $I(P)$ is defined as in \eqref{aswedidforthree}.
We evaluate this sum in the same way that we evaluated $a_3'(p^m)$ in Appendix \ref{appxtwoandthree}:
We expand the product in the summand and eliminate the symbols $I(P)$ by splitting the sum into several geometric series over different ranges, corresponding to the conditions $j_2>0$, and so on, and compute each geometric series individually. We assume $m\geq 1$ to guarantee that $\sum_{j_2=1}^m$, for instance, is never an empty sum. Thus, by a tedious but straightforward calculation, we get 
\begin{gather}
    a_4'(p^m) = \dfrac{(p-1) p^{m-6}}{p+1}
    \Big(p^{2 m}-p^{m+1}-4 p^{m+2}-6 p^{m+3}-4 p^{m+4}-p^{m+5}+3 p^{2 m+1}+\\
    6 p^{2 m+2}+7 p^{2m+3}+6 p^{2 m+4}+3 p^{2 m+5}+p^{2 m+6}+p^3\Big). \label{fusk1}
\end{gather}

Using this, one may show that
\begin{gather}
    a_4'(p^{m+1})^2-a_4'(p^m)a_4'(p^{m+2})=\\
    (p-1)^4 p^{3 m-7} \left((p+1)^2 \left(p^2+p+1\right)^3 p^{2 m}-\left(p^2+p+1\right)^3 p^m+(p+1)^2 p\right) = \label{fusk2} \\
    (p-1)^4 p^{3 m-7} \left( ((p+1)^2 p^m-1)\left(p^2+p+1\right)^3 p^{m}+(p+1)^2 p\right),
\end{gather}
which we see is positive for all $p\geq 2$ and all $m\geq 1$. Moreover, using $a_4'(p^0)=1$, we get
\begin{gather}
        a_4'(p^{1})^2-a_4'(p^0)a_4'(p^{2}) = (p-1) (p+2) \left(p^3-3\right),
\end{gather}
which is positive for all $p\geq 2$.
Thus we have proved the inequality \eqref{anprimelogconcave} for all $m\geq 0$, which completes the proof of Proposition \ref{fourandfive} for the case $n=4$.\qed

Equations \eqref{fusk1} and \eqref{fusk2} may be verified with a computer algebra system, for instance with the Mathematica code provided at
\begin{center}
    \url{http://www.math.kth.se/~holmin/files/x/a4prime_is_logconcave}.
\end{center}

\subsection{The case $n=5$}
We prove Proposition \ref{fourandfive} for $n=5$. We repeat the procedure above. We evaluate
\begin{gather}
    a_5'(p^m)=
    \sum_{j_2+j_3+j_4+j_5=m}\left(p^{j_2}-p^{j_2-1}I(j_2>0)\right)\left(p^{2j_3}-p^{2(j_3-1)}I(j_3>0)\right)\times\\
    \times \left(p^{3j_4}-p^{3(j_4-1)}I(j_4>0)\right)\left(p^{4j_5}-p^{4(j_5-1)}I(j_4>0)\right).
\end{gather}
As before, we expand the product in the summand, and split the sum into several geometric series. This yields
$a_5'(p^m) = \frac{p-1}{p^{10} (p+1) \left(p^2+p+1\right)} \Big(
    p^{4 m}-p^{m+6}+p^{2 m+3}+5 p^{2 m+4}+11 p^{2 m+5}+14 p^{2 m+6}+11 p^{2 m+7}+5 p^{2
   m+8}+p^{2 m+9}-p^{3 m+1}-5 p^{3 m+2}-15 p^{3 m+3}-30 p^{3 m+4}-45 p^{3 m+5}-51
   p^{3 m+6}-45 p^{3 m+7}-30 p^{3 m+8}-15 p^{3 m+9}-5 p^{3 m+10}-p^{3 m+11}+4 p^{4
   m+1}+10 p^{4 m+2}+20 p^{4 m+3}+31 p^{4 m+4}+40 p^{4 m+5}+44 p^{4 m+6}+40 p^{4
   m+7}+31 p^{4 m+8}+20 p^{4 m+9}+10 p^{4 m+10}+4 p^{4 m+11}+p^{4 m+12}
   \Big)$, valid for $m\geq 1$.

We get $a_5'(p)-a_5'(1)a_5'(p^2)=(p-1) \left((p-1) p \left(p^2+p+3\right) (p (p+2)+2)-10\right)$, which we see is positive, and thus we have proved the inequality \eqref{anprimelogconcave} for $m=0$.

For $m\geq 1$, we get $a_5'(p^{m+1})^2-a_5'(p^m)a_5(p^{m+2})=
   \frac{(p-1)^4 p^{3 m-13}}{p^2+p+1} \Big(
    -p^{3 m}+p^{4 m}-p^{m+2}-4 p^{m+3}-10 p^{m+4}-16 p^{m+5}-19 p^{m+6}-16 p^{m+7}-10
   p^{m+8}-4 p^{m+9}-p^{m+10}+2 p^{2 m+1}+10 p^{2 m+2}+34 p^{2 m+3}+80 p^{2 m+4}+143
   p^{2 m+5}+201 p^{2 m+6}+224 p^{2 m+7}+201 p^{2 m+8}+143 p^{2 m+9}+80 p^{2 m+10}+34
   p^{2 m+11}+10 p^{2 m+12}+2 p^{2 m+13}-8 p^{3 m+1}-32 p^{3 m+2}-88 p^{3 m+3}-188
   p^{3 m+4}-328 p^{3 m+5}-480 p^{3 m+6}-600 p^{3 m+7}-646 p^{3 m+8}-600 p^{3
   m+9}-480 p^{3 m+10}-328 p^{3 m+11}-188 p^{3 m+12}-88 p^{3 m+13}-32 p^{3 m+14}-8
   p^{3 m+15}-p^{3 m+16}+6 p^{4 m+1}+23 p^{4 m+2}+64 p^{4 m+3}+143 p^{4 m+4}+266 p^{4
   m+5}+423 p^{4 m+6}+584 p^{4 m+7}+706 p^{4 m+8}+752 p^{4 m+9}+706 p^{4 m+10}+584
   p^{4 m+11}+423 p^{4 m+12}+266 p^{4 m+13}+143 p^{4 m+14}+64 p^{4 m+15}+23 p^{4
   m+16}+6 p^{4 m+17}+p^{4 m+18}+p^6+2 p^5+p^4
   \Big)$. The first factor is obviously positive for $p\geq 2$, and the second factor may be rearranged as
$
(752 p^{4m+9}        -646 p^{3m+8}  )+             
(706 p^{4m+10}       -600 p^{3m+9}  )+             
(706 p^{4m+8}        -600 p^{3m+7}  )+             
(584 p^{4m+11}       -480 p^{3m+10} )+             
(584 p^{4m+7}        -480 p^{3m+6}  )+             
(423 p^{4m+12}       -328 p^{3m+11} )+             
(423 p^{4m+6}        -328 p^{3m+5}  )+             
(266 p^{4m+13}       -188 p^{3m+12} )+             
(266 p^{4m+5}        -188 p^{3m+4}  )+             
(143 p^{4m+14}       -88 p^{3m+13}  )+             
(143 p^{4m+4}        -88 p^{3m+3}   )+             
(64 p^{4m+15}        -32 p^{3m+14}  )+             
(64 p^{4m+3}         -32 p^{3m+2}   )+             
(23 p^{4m+16}        -8 p^{3m+15}   )+             
(23 p^{4m+2}         -8 p^{3m+1}    )+             
(6 p^{4m+17}         -  p^{3m+16}   )+             
(6 p^{4m+1}          -  p^{3m}      )+             
(224 p^{2m+7}        -19 p^{m+6}   )+             
(201 p^{2m+8}        -16 p^{m+7}   )+             
(201 p^{2m+6}        -16 p^{m+5}   )+             
(143 p^{2m+9}        -10 p^{m+8}   )+             
(143 p^{2m+5}        -10 p^{m+4}   )+             
(80 p^{2m+10}        -4 p^{m+9}    )+             
(80 p^{2m+4}         -4 p^{m+3}    )+             
(34 p^{2m+11}        -  p^{m+10}   )+             
(34 p^{2m+3}         -  p^{m+2}    )              
+10 p^{2m+12}                                      
+10 p^{2m+2}                                       
+2 p^{2m+13}                                       
+2 p^{2m+1}                                        
+2 p^5                                             
+  p^{4m+18}                                       
+  p^{4m}                                          
+  p^6                                             
+  p^4,$
where every term is positive for all $p\geq 2$, and we have thus proved the inequality \eqref{anprimelogconcave} for $m\geq 1$. This concludes the proof of Proposition \ref{fourandfive} for $n=5$.\qed

The computations of $a_5'(p^m)$ and $a_5'(p^{m+1})^2-a_5'(p^m)a_5'(p^{m+2})$ may be verified with the Mathematica code provided at
\begin{center}
\noindent\url{http://www.math.kth.se/~holmin/files/x/a5prime_is_logconcave}.
\end{center}

\section{Calculation of a measure} \label{section_katznelson_measure}\index{$C_0$}
In \cite{Katznelson} the asymptotics
\[N_{n,0}(T)=\dfrac{n-1}{\zeta(n)}w(B)T^{n(n-1)}\log T+\BigO(T^{n(n-1)})\]
are given, where $B$ is the unit ball in $M_n(\R)$. 
The measure $w$ on $M_n(\R)$ is defined in \cite{Katznelson} as follows.
Let $A_u\defeq \{A\in M_n(\R): Au=0\}$ be the space of matrices annihilating the nonzero vector $u\in\R^n\setminus\{0\}$. We define for (Lebesgue measurable) subsets $E\subseteq M_n(\R)$ the measure $w_u(E)\defeq \vol(E\cap A_u)$ where $\vol$ is the standard $n(n-1)$-dimensional volume on $A_u$, and define the measure $w(E)\defeq (1/2)\int_{\operatorname S^{n-1}}w_u(E)\dd \nu(u)$, where $\nu$ is the standard Euclidean surface measure on the $(n-1)$-dimensional sphere $\operatorname S^{n-1}$.

We shall now calculate $w(B)$.
The set $B\cap A_u$ is the unit ball in the $n(n-1)$-dimensional vector space $A_u$. Its volume does not depend on $u\neq 0$, and if $u=(0,\ldots,0,1)$, then $B\cap A_u$ is the unit ball in $\R^{n(n-1)}$, when identifying $M_n(\R)$ with $\R^{n^2}$. Denote by $V_{n(n-1)}$ the volume of the unit ball in $\R^{n(n-1)}$. Thus $w_u(B)=V_{n(n-1)}$, independently of $u\neq 0$, and
\[w(B)=V_{n(n-1)}\dfrac12\int_{\operatorname S^{n-1}}\dd \nu(u)=\dfrac{V_{n(n-1)}S_{n-1}}2,\]
where $S_{n-1}$ is the surface area of the sphere $\operatorname S^{n-1}$. The volume and surface area of the unit ball is well known, and we may explicitly calculate
\[C_0\defeq w(B)=\dfrac{\pi^{n^2/2}}{\Gamma\left(\dfrac n2\right)\Gamma\left(\dfrac{n(n-1)}2+1\right)}.\]
Recalling from Theorem \ref{drs_lemma} the expression for $C_1$, we get the following relation.
\[C_1=\dfrac1{\zeta(2)\cdots\zeta(n)}\dfrac{\pi^{n^2/2}}{\Gamma\left(\dfrac n2\right)\Gamma\left(\dfrac{n(n-1)}2+1\right)}=\dfrac1{\zeta(2)\cdots\zeta(n)}C_0.\]

\section*{Acknowledgements}
 I would like to thank my advisor Pär Kurlberg for suggesting this problem to me and for all his help and encouragement.

\phantomsection
\addcontentsline{toc}{chapter}{Bibliography}
\bibliographystyle{alpha} 
\bibliography{mybib}{}


\end{document}